\documentclass[12pt, a4paper]{article}
\usepackage[utf8]{inputenc} 
\usepackage[T1]{fontenc}
\usepackage[english]{babel} 
\usepackage{amsmath, amssymb, amsthm, amsfonts} 
\usepackage{graphicx} 
\usepackage{geometry} 
\usepackage{authblk}
\usepackage{float}
\usepackage{caption}
\usepackage{subcaption}
\usepackage{cite}
\newtheorem{theorem}{Theorem}[section]
\newtheorem{lemma}[theorem]{Lemma}
\newtheorem{definition}[theorem]{Definition}
\newtheorem{remark}[theorem]{Remark}
\newtheorem{corollary}{Corollary}[section]

\title{Almost Periodicity in  Exponential Dichotomous Linear Dynamics with Piecewise Constant Argument}
\author[1]{Marat Akhmet\thanks{marat@metu.edu.tr}}
\author[1]{Furkan Duraner\thanks{furkan.duraner@metu.edu.tr}}

\affil[1]{Department of Mathematics, Middle East Technical University}

\date{}

\begin{document}
\allowdisplaybreaks

\maketitle

\begin{abstract} The paper examines an inhomogeneous system of differential equations with generalized piecewise constant arguments. The focus of the research is on the almost periodicity of both the perturbation and the solution. The system demonstrates exponential dichotomy, indicating complex asymptotic behavior among neighboring motions. A notable methodological aspect of the analysis is the use of Green's function to construct a unique almost periodic solution. Additionally, a comparison with previous results has been made. An illustrative example with numerical visualizations is provided to clarify this concept.
\end{abstract}

\section{Introduction and Preliminaries}
\indent A new class of differential equations, known as Equations with Piecewise Constant Argument of Generalized type (EPCAG), was introduced and systematically developed in a series of works \cite{1,2,3,4,5,6,7,8,9,10,11,12}. These systems  provide a flexible framework for modeling hybrid phenomena in which continuous dynamics interact with discrete mechanisms based on piecewise  memory. EPCAG models have found applications in neuroscience, control theory, and population dynamics.\\
Almost periodicity is considered the most developed concept for recurrent dynamics. Despite its rather theoretical merits, it attracts researchers with the completeness of functional properties and challenges to prove the conditions for outputs of models under investigation.  The existence of almost periodic solutions remains one of the most compelling topics in the theory of differential equations \cite{21,22,23}. The systematic investigation of almost periodicity for systems with generalized piecewise constant argument  was pioneered in \cite{3}.\\
Exceptional difficulties, one meets if considers hybrid systems, when the property has to be observed in continuous and discrete time, and consistent for both of them \cite{2}.  Another intriguing task in analyzing hybrid systems is the description of exponential dichotomies. In our papers \cite{2,7,1}, we have introduced this concept and provided constructive circumstances that are effective for practical applications.\\
\indent Within this established framework, the present paper investigates the existence and construction of almost periodic solutions to nonhomogeneous linear EPCAG systems whose homogeneous part admits exponential dichotomy. Our analysis refines and further develops the Green-function approach introduced in earlier works \cite{6,7,10} together with improved Cauchy matrix representations and integral formulas adapted to the EPCAG structure.\\
\indent We show that if the nonhomogeneous term is almost periodic, then, under an exponential dichotomous homogeneous part, the entire EPCAG system possesses a unique almost periodic solution. These results extend the existing existence-uniqueness theory to a broader class of inputs and hybrid configurations. Also, an example and a numerical simulation are provided to illustrate exponential dichotomy and almost periodic behavior within a unified theoretical framework.\\
\indent \indent Following the foundational developments in \cite{1,2,3,4,5,6,7,8,9,10,11,12}, it is well established that EPCAG provides a robust framework for modeling hybrid dynamics. Building upon these established principles, the present work focuses on a qualitative analysis of the following nonhomogeneous linear system,
\begin{equation} \label{eq:second}
	z'(t) = A_0(t)z(t) + A_1(t)z(\gamma(t)) + f(t),
\end{equation}
where $t$ is a real number, $z(t)$ is an $n$-dimensional real-valued vector function, and the coefficients $A_0(t), A_1(t)$ are square $n$-dimensional real-valued matrix-functions, continuous and bounded. The forcing term $f(t)$ is a continuous and bounded $n$-dimensional real-valued vector-function. The entries of the system and the argument function $\gamma(t)$ will be completely defined in the preliminaries provided in Section \ref{s2}.\\
\indent To investigate the almost periodicity for the system \eqref{eq:second}, we will apply constructive conditions on the associated homogeneous system,
\begin{equation} \label{eq:first}
	z'(t) = A_0(t)z(t) + A_1(t)z(\gamma(t)).
\end{equation}
\indent Throughout this paper, a fundamental assumption is that system \eqref{eq:first} admits \textit{exponential dichotomy}. Under the condition, we aim to establish the existence and uniqueness of a Bohr almost periodic solution for the main system \eqref{eq:second}.\\
\indent A primary focus of this study is to demonstrate that almost periodicity constitutes a significant asymptotic structure for the solutions of EPCAG systems. Our methodology relies on the  equivalent integral equation  adapted to the EPCAG framework \cite{12,20}, combined with new estimates for the specifically defined Green's function, which were initially determined in \cite{2,11}. By using these estimations, we evaluate the diagonal almost periodicity of the components which provides the necessary analytical foundation to prove the existence of the almost periodic solution of system \eqref{eq:second}.\\
\indent It should be mentioned that in the recent work of Chiu \cite{24} related to the present research  the critical analytical flaws found. In the aforementioned study, the author attempts to establish the existence of almost periodic solutions by reducing the continuous-time EPCAG system to a discerete-time difference equation. This reduction-based approach relies exclusively on the exponential dichotomy of the discerete system, following the classical framework of Wiener \cite{16}. However, the investigation of consistency of the dichotomies of  the discrete reduction and the original continuous evolution is regrettably absent in \cite{24}. It is well established in the theory of hybrid systems that the exponential dichotomy of the discrete system does not automatically imply the existence of uniform exponential bounds for the continuous-time model. Without a rigorous proof of the compatibility conditions that link the exponential estimates of the discrete system to the inter-interval behavior, any claim regarding the continuous-time almost periodicity remains mathematically incomplete.\\
\indent Furthermore, the methodology in \cite{24} fails to provide the necessary conditions under which the exponential dichotomy of the reduced discrete system can be lifted to the continuous system. In EPCAG framework, the argument $\gamma(t)$ introduces a unique coupling between the discrete state and the continuous derivative. Consequently, establishing an almost periodic solution for the reduced discrete system is merely a preliminary step that does not, by itself, verify that the solution of the differential equation satisfies the required Bohr-periodicity. The present paper develops the almost periodicity theory directly for the continuous EPCAG system. By grounding our analysis in the exponential dichotomy of the continuous fundamental matrix, we eliminate the ambiguities inherent in discrete reductions. Our approach ensures that the dichotomic separation is maintained throughout the entire real line, not just at isolated discrete points. \\
\indent Moreover, \cite{24} contains a significant conceptual contradiction regarding global stability. The author asserts that a nonhomogeneous equation with an exponential dichotomous linear part admits a globally exponentially stable almost periodic solution. This assertion is theoretically untenable; by definition, an exponential dichotomy implies the existence of a non-trivial unstable manifold. The presence of such a divergent subspace inherently precludes the possibility of global exponential stability for any bounded solution. Such fundamental discrepancies in the underlying stability assumptions of \cite{24} potentially undermine the analytical validity of the subsequent results.\\  
\indent To rectify the inconsistencies such as Theorem 4, Theorem 9 in \cite{24}, this study provides a mathematically sound framework where the exponential dichotomy is verified directly for the continuous-time system. This ensures that the resulting almost periodic solution is established under consistent and minimal assumptions, providing a reliable bridge between the theory of hybrid systems and asymptotic analysis. To illustrate the superiority of this direct approach, we present an example where the exponential dichotomic separation in the continuous system is numerically evaluated, confirming that our solutions are robust and analytically valid.\\
\indent To ensure a rigorous analysis, we provide the following fundamental definitions that characterize the solutions and the almost periodic structure of the system.\\
\indent Let $\mathbb{Z}$, $\mathbb{N}$, and $\mathbb{R}$ be the sets of integers, natural numbers, and real numbers, respectively. We denote by $\mathbb{R}^n$ the $n$-dimensional real vector space and by $\mathbb{R}^{n \times n}$ the space of all $n \times n$ real matrices. For any vector $z \in \mathbb{R}^n$, its Euclidean norm is represented by $\|z\|$, while for a matrix $A \in \mathbb{R}^{n \times n}$, we denote its induced matrix norm similarly by $\|A\|$.
\begin{definition}\cite{2} \label{def:epcag}Let $\theta = \{\theta_i\}$ and $\zeta = \{\zeta_i\}$, $i \in \mathbb{Z}$, be two fixed real-valued sequences satisfying the following conditions:
\begin{enumerate}
\item[(i)] $\theta_i < \theta_{i+1}$ for all $i \in \mathbb{Z}$, and $|\theta_i| \to \infty$ as $|i| \to \infty$;
\item[(ii)] $\theta_i \leq \zeta_i < \theta_{i+1}$ for all $i \in \mathbb{Z}$.
\end{enumerate}
The argument function $\gamma(t)$, which characterizes the piecewise constant nature of the systems \eqref{eq:second} and \eqref{eq:first}, is defined by
\begin{equation}\gamma(t) = \zeta_i, \quad \text{if } t \in [\theta_i, \theta_{i+1}), \quad i \in \mathbb{Z}.
\end{equation}
\end{definition}
\begin{definition}\cite{2} A continuous function $ z(t) $ is a solution of \eqref{eq:second} (or \eqref{eq:first}) on $ \mathbb{R} $ if:
\begin{enumerate} \label{solndef}
\item the derivative $ z'(t) $ exists at each point $ t \in \mathbb{R} $, with the possible exception of the points $ \theta_i $, $ i \in \mathbb{Z} $, where one-sided derivatives exist;
\item the equation is satisfied by $ z(t) $ on each interval $ (\theta_i, \theta_{i+1}) $, $ i \in \mathbb{Z} $, and it holds for the right derivative of $ z(t) $ at the points $ \theta_i $, $ i \in \mathbb{Z} $.
\end{enumerate}\label{def:solution}
\end{definition}
\indent Denote by $X(t,s)$, $X(s,s)=I$, where $I$ is the $n\times n$ identity matrix, $t,s \in \mathbb{R}$, the fundamental matrix of solutions of the system,
\begin{equation}
    x'(t)=A_0(t)x(t),
\end{equation}
which is associated with systems  \eqref{eq:second} and \eqref{eq:first}.\\
\indent We introduce the following matrix function which is essential for characterizing the transition of the state $z(t)$ between the switching moments  $\theta_i$ and $\theta_{i+1}$ for each index $i$.
\begin{equation}
M_i(t)=X(t,\zeta_i)+\int\limits_{\zeta_i}^tX(t,s)A_1(s)ds.
\end{equation}
\begin{definition} \cite{1} Let a set $C_0(\mathbb{R})$ be a set of all bounded and uniformly continuos functions on $\mathbb{R}$. For a function $f\in C_0(\mathbb{R})$ and $\tau\in \mathbb{R}$, \textit{a translate by $\tau$} is the function $f(t+\tau)$.
\end{definition}
\begin{definition} \cite{1} A number $\tau \in \mathbb{R}$ is called an \textit{$\epsilon$-translation number of $f$} $\in C_0(\mathbb{R})$ if the translation of $f$ by $\tau$ satisfies 
\begin{equation*}
    \|f(t+\tau) - f(t)\| < \epsilon \quad\text{for every } t\in \mathbb{R}.
\end{equation*}
\end{definition}
\begin{definition}\cite{23} \label{def:relatively_dense}
A set $S \subset \mathbb{R}$ is said to be \textit{relatively dense} if there exists a positive number $L > 0$ such that every interval of length $L$ contains at least one element of $S$. That is, for every $a \in \mathbb{R}$,
\begin{equation*}
[a, a+L] \cap S \neq \emptyset.
\end{equation*}
\end{definition}
\begin{definition} \cite{23} A function $f \in C_0(\mathbb{R})$ is said to be \textit{almost periodic} (in the sense of Bohr) if for every $\epsilon > 0$, there exists a relatively dense set of $\epsilon$-translation numbers $\tau$ of $f$.
\end{definition}
\begin{definition}\cite{23}
    A matrix function $A(t) \in \mathbb{R}^{n \times n}$ is said to be \textit{almost periodic} (in the sense of Bohr) if for every $\varepsilon > 0$, there exists a relatively dense set of $\varepsilon$-translation numbers $\tau \in \mathbb{R}$ such that the following inequality holds:
\begin{equation*}
|| A(t + \tau) - A(t) || < \varepsilon, \quad \text{for all } t \in \mathbb{R}.
\end{equation*}
\end{definition}
\begin{definition}\cite{23} \label{def:ap_sequence}
A sequence of real or complex $k$-dimensional vectors $a_i$,$i \in
\mathbb{Z}$, is said to be almost periodic if for arbitrary $\epsilon>0$, there exists a relatively dense set of integers $q$ which satisfy the inequality 
\begin{align*}
||a_{i+q} - a_i||  < \epsilon \quad \text{for all}\quad i \in \mathbb{Z}
\end{align*}
The number $q$ is said to be an $\epsilon$-almost period of the sequence.
\end{definition}
The sequences $\theta$ and $\zeta$ in Definition \ref{def:epcag} play a critical role in the almost periodic behavior of the system. Let us denote   $\theta_i^j = \theta_{i+j} - \theta_i$ and $\zeta_i^j = \zeta_{i+j} - \zeta_i$ for all $i, j \in \mathbb{Z}$
.
\begin{definition} The family of sequences $\{\theta_i^j\}$, $j\in \mathbb{Z}$, is called \textit{equipotentially almost periodic} \cite{12, 15} if for an arbitrary $\epsilon > 0$ there exists a relatively dense set of $\epsilon$-translation numbers, common for all sequences $\{\theta_i^j\}$, $j\in \mathbb{Z}$.\label{def:unif_ap}
\end{definition}
\section{Solvability of the Model and the Fundamental Matrix} \label{s2}
\indent \indent The following assumptions will be needed throughout this paper:\vspace{0,2cm}\\
\indent To ensure that the entire dynamic structure including the continuous coefficients, external force and the discrete switching sequences remains consistent across the entire real line, we assume the following almost periodicity properties\vspace{0,1cm}\\
\noindent \textbf{(C1)} The matrices $A_0(t), A_1(t)\in \mathbb{R}^{n\times n}$ and the vector function $f(t) \in \mathbb{R}^n$ are almost periodic; \vspace{0,1cm}\\
\noindent \textbf{(C2)} The sequences $\{\zeta_i^j\}$ and $\{\theta_i^j\}$ are equipotentially almost periodic for all $j \in \mathbb{Z}$; \vspace{0,1cm}\\
\indent These conditions are required to synchronize the behavior of the nonhomogeneous term with the timing of the switching moments. Specifically, (C2) guarantees that the discrete jumps do not disrupt the almost periodic nature of the solution, allowing the system to maintain its character throughout the entire time domain. Moreover (C2) implies that $|\theta_i|, |\zeta_i| \to \infty$ as $|i| \to \infty$.\vspace{0,1cm}\\
\indent To ensure that the switching moments are properly spaced across the entire real line and to maintain a consistent temporal structure, we assume the following bounds\vspace{0,1cm}\\
 \noindent \textbf{(C3)} there exist positive numbers $\underline{\theta}$ and $\bar{\theta}$ such that $\underline{\theta} \leq\theta_{i+1}-\theta_i \leq \bar{\theta}, \forall i \in \mathbb{Z}$;\vspace{0,1cm}\\
 \indent Condition (C3) regulates the spacing of the switching sequence by maintaining the length of each interval within the bounds $[\underline{\theta}, \bar{\theta}]$. This prevents the intervals from becoming excessively large while ensuring that the switching moments do not cluster together, thereby providing a consistent discrete structure for the solution across the entire real line.\vspace{0,1cm}\\
 \indent To guarantee that the solution is uniquely determined and remains invertible within each switching interval across the entire real line, we assume the following non-singularity condition\vspace{0,1cm}\\
 \noindent \textbf{(C4)} For every fixed $i\in \mathbb{Z}$, $\text{det}[M_i(t)]\neq 0$, $\forall t\in [\theta_i,\theta_{i+1}]$;\vspace{0,1cm}\\
  \indent This condition is necessary to ensure the local existence and uniqueness of the solution.\vspace{0,1cm}\\
\indent In this study, we assume that the solutions of the given equations are continuous functions on $\mathbb{R}$ since they satisfy the condition $z(\theta_i) = z(\theta_i-)$ for each index $i$, where $z(\theta_i-) = \lim\limits_{t \to \theta_i^-} z(t)$ denotes the left-hand limit of $z(t)$ at $t = \theta_i$. While the solutions themselves are continuous, the piecewise constant argument $\gamma(t)$ is discontinuous. This causes the derivatives of solutions and the right-hand sides of \eqref{eq:second} and \eqref{eq:first} to generally exhibit discontinuities at the switching moments $\theta_i$ for each index $i$. \\
\indent Hence, we consider the solutions of these equations as functions that remain continuous on $\mathbb{R}$ and continuously differentiable within each interval $(\theta_i,\theta_{i+1})$ for each index $i$. It should be noted that the derivatives $z'(t)$ may possess discontinuities at the points $\theta_i$, where only one-sided derivatives are guaranteed to exist.
\begin{theorem}\cite{2}
Assume that $A_0(t), A_1(t) \in \mathbb{R}^{n \times n}$ are continuous. Then for every $(t_0,z_0)\in \mathbb{R}\times\mathbb{R}^n$, there exists a unique solution $z(t)$ of \eqref{eq:first} in the sense of Definition \ref{def:solution} such that $z(t_0)=z_0$ if and only if the condition (C4) is valid. \label{thm1}
\end{theorem}
This theorem establishes a correspondence between the points $(t_0,z_0)\in \mathbb{R}\times\mathbb{R}^n$ and the solutions of \eqref{eq:first}, ensuring that no solution of the equation exists outside of this correspondence. Using this assertion, we can say that the definition of the initial value problem for differential equations with piecewise constant arguments is similar to that for ordinary differential equations. In particular, the dimension of the space of all solutions is $n$. Hence, the investigation of
problems considered in this section does not need to be supported by the results of the theory of functional differential equations \cite{3, 13}, despite the fact EPCAG are equations with deviated arguments. \\
\indent The system \eqref{eq:first} is a differential equation with a deviated argument. That is why, it is reasonable
to assume that the interval $[\theta_i,\theta_{i+1})$, for an index $i$, containing $t_0$ must consist of more than one point as a functional differential equation. The following arguments show that, in our case, we need only one initial moment. In fact, assume that $(t_0,z_0)$ is fixed, and there exists an index $i$ such that $\theta_i \leq t_0 < \theta_{i+1}$. The solution satisfies, on the interval $[\theta_i,\theta_{i+1})$ $i\in\mathbb{Z}$, the functional differential equation
\begin{equation*}
    z'(t)=A_0(t)z(t)+A_1(t)z(\zeta_i).
\end{equation*}
\indent Assume that $t_0 \neq \zeta_i$. Then, formally, we need the pair of initial points $(t_0,z_0)$ and $(\zeta_i,z(\zeta_i))$ to proceed with the solution. Indeed, since $z_0=M_i(t_0)z(\zeta_i)$ as explained in \cite{2} and the matrix $M_i(t_0)$ is nonsingular, we can say that the initial condition $z(t_0) = z_0$ is sufficient to define the solution.\\
\indent Hence, for a fixed $t_0\in \mathbb{R}$ there exists a fundamental matrix of solutions of \eqref{eq:first},\\ $Z(t)=Z(t,t_0)$, $Z(t_0,t_0)=I$ such that 
\begin{equation*}
    \frac{dZ}{dt}=A_0(t)Z(t)+A_1(t)Z(\gamma(t))
\end{equation*}
Without loss of generality assume that $\theta_i < t_0 < \zeta_i$ for a fixed index $i$ and define, as it is done in \cite{1}, the following matrix for increasing t
\begin{equation}\label{eqn7}
    Z(t)=M_l(t)\left[ \prod\limits_{k=l}^{i+1}M_k^{-1}(\theta_k)M_{k-1}(\theta_k) \right]M_i^{-1}(t_0), 
\end{equation}
if $t\in[\theta_l,\theta_{l+1})$, for arbitrary $l>i$.\\
\indent Similarly, if $\theta_j \leq t \leq \theta_{j+1}<...<\theta_i \leq t_0 \leq \theta_{i+1}$
\begin{equation}
    Z(t)=M_j(t)\left[ \prod\limits_{k=j}^{i-1}M_k^{-1}(\theta_{k+1})M_{k+1}(\theta_{k+1}) \right]M_i^{-1}(t_0).
\end{equation}
\indent One can easily see that 
\begin{equation}
Z(t,s)=Z(t)Z^{-1}(s), \quad t,s \in \mathbb{R}
\end{equation}
   and a solution $z(t),z(t_0)=z_0, (t_0.z_0) \in \mathbb{R}\times\mathbb{R}^n$ of \eqref{eq:first} is equal to 
   \begin{equation}\label{eqn10}
       z(t)=Z(t,t_0)z_0 \quad t\in \mathbb{R.} 
   \end{equation}
The last formulas, (6)-(9), have been obtained in \cite{2}.\\
\indent Under the conditions (C1),(C3) and (C4), it can be shown that the transition matrices $Z(t,s)$ and $X(t,s)$ are uniformly bounded on each interval $[\theta_i,\theta_{i+1})$ for the indices $i$ and $t,s \in \mathbb{R}$.\\
\indent It is important to note that, since the matrices $A_0(t), A_1(t) \in \mathbb{R}^{n \times n}$ are almost periodic according to (C1), they are necessarily uniformly bounded.\\
\indent Consequently, there exist positive constants $m,M,\bar{M}$ such that $m\leq ||Z(t,s)|| \leq M$, and $||X(t,s)||\leq\bar{M}$, if $t,s \in [\theta_i,\theta_{i+1})$ for the indices $i$; these bounds will play a crucial role for the lemmas in the following sections.
\section{The Green Function for the Exponentially Dichotomous System}
\indent \indent In the various theories of differential and discrete equations, exponential dichotomy serves as a fundamental generalization of hyperbolicity to non-autonomous systems. This property facilitates the partitioning of the state space into stable and unstable subspaces, thereby establishing a fundamental framework for investigating the boundedness and stability of solutions. \\
\indent To characterize the stability properties of the system \eqref{eq:first}, denote by $Z(t)$, $t \in \mathbb{R}$ the Cauchy matrix satisfying $Z(t_0)=I$ for some fixed $t_0 \in \mathbb{R}$. Let $P$ be a projection matrix associated with system \eqref{eq:first}, i.e., $P^2=P$, $P\in\mathbb{R}^{n \times n}$.\\
\indent Then define
\begin{align}
    Z_-(t,s) := Z(t)PZ^{-1}(s), \qquad
    Z_+(t,s) := Z(t)(I-P)Z^{-1}(s). \label{EXPDICH}
\end{align} 
\indent The system \eqref{eq:first} is said to admit an exponential dichotomy on $\mathbb{R}$ if there exist positive constants $\sigma_1, \sigma_2, K_1, K_2$ such that the following estimates hold:
\begin{align}
    \|Z_-(t,s)\| &\leq K_1e^{-\sigma_1(t-s)}, \qquad t\geq s, \label{eq:Zminus} \\
    \|Z_+(t,s)\| &\leq K_2e^{\sigma_2(t-s)}, \qquad t\leq s. \label{eq:Zplus}
\end{align}
\indent To construct the Green's function for the differential equation with generalized piecewise constant argument, we must carefully manage the interplay between continuous flows and discrete transitions.\\
\indent It is important to note that for any arbitrary $t \in \mathbb{R}$, there exists a unique integer $j$ such that $t$ lies within the interval $[\theta_j, \theta_{j+1})$.\\
\indent In the classical theory of exponential dichotomies, the Green's function is fundamentally divided into two distinct components based on the order of $t$ and $s$: the stable evolution, which governs the trajectories when $s < t$, and the unstable evolution, which governs the trajectories when $s \geq t$.\\
\indent As explicitly defined in \cite{1}, these fundamental components are formulated as:
\begin{align}
    \Psi_-(t, s) &= 
    \begin{cases}
        Z_-(t, \theta_{k+1}) X(\theta_{k+1}, s), & s \in [\zeta_k, \zeta_{k+1}], \\
        X(t) P X^{-1}(s), & s \in \widehat{[\zeta_j, t]},
    \end{cases}\quad s < t, \label{psiminus}\\
    \Psi_+(t, s) &= 
    \begin{cases}
        Z_+(t, \theta_{k+1}) X(\theta_{k+1}, s), & s \in [\zeta_k, \zeta_{k+1}], \\
        X(t)(I - P) X^{-1}(s), & s \in \widehat{[\zeta_j, t]},
    \end{cases}\quad s \geq t\label{psiplus}
\end{align} 
where $t \in [\theta_j,\theta_{j+1})$, and the notation $\widehat{[a,b]}$ denotes the interval $[a,b]$ if $a \leq b$, and $[b,a]$ if $a > b$.\\ 
\indent Using these notations, the Green's function is defined in \cite{1} as:
\begin{align}
    G(t,s) = 
    \begin{cases}
        \Psi_-(t,s), & s < t, \\
        \Psi_+(t,s), & s \geq t.         
    \end{cases}
\end{align}
\indent To provide a more compact formulation, we redefine the Green's function $G(t, s)$ in this paper directly based on the location of $s \in [\zeta_k, \zeta_{k+1}]$ relative to the interval index $j$ associated with $t$:
\begin{align}\label{G(t,s)}
    G(t,s) = 
    \begin{cases}
        Z_-(t, \theta_{k+1}) X(\theta_{k+1}, s), & s \in [\zeta_k, \zeta_{k+1}], \quad k+1 \leq j \\
        Z_+(t, \theta_{k+1}) X(\theta_{k+1}, s), & s \in [\zeta_k, \zeta_{k+1}], \quad k \geq j \\
        X(t,s), \quad s \in \widehat{[\zeta_j,t]}
    \end{cases}
\end{align}
\indent Note that the index condition $k+1 \leq j$ restricts $s$ to the intervals  $s < t$, while $k \geq j$ corresponds to the case $s \geq t$.\\
\indent Under the exponential dichotomy assumption of the equation \eqref{eq:first}, we consider the following integral representation to investigate the unique bounded solution of the non-homogeneous equation \eqref{eq:second}:
\begin{align}
    z(t) = \int\limits_{-\infty}^{\infty} G(t,s)f(s)\,ds. \label{eqnG}
\end{align}
\indent To streamline the analysis, we derive an exponential estimate for the Green’s function \eqref{G(t,s)} by accounting for the transition through jump moments. Specifically, this estimate incorporates the discrete switching dynamics encountered as the solution evolves from $s$ toward the current continuity interval of $t$. By letting $\sigma=\min \{\sigma_1,\sigma_2\}$ and $K=\max\{K_1 \bar{M} e^{\sigma_1 \bar{\theta}}, K_2 \bar{M} e^{\sigma_2 \bar{\theta}},\bar Me^{\sigma\bar{\theta}}\}$, we can effectively bridge the gap between the discrete switching moments and the continuous evolution of the system. This relationship is formally established in the following lemma.
\begin{lemma} \label{lemma3.2}
    If system \eqref{eq:first} admits an exponential dichotomy on $\mathbb{R}$, then the  Green's function \eqref{G(t,s)}  satisfies the following exponential estimate:
     \begin{align}
    ||G(t,s)|| \leq Ke^{-\sigma|t-s|}, \quad\forall t,s \in \mathbb{R}\label{||G(t,s)||}
\end{align}
\end{lemma}
\begin{proof}
    The proof directly follows from the definition of the Green's function \eqref{G(t,s)}, and the exponential dichotomy estimates \eqref{eq:Zminus}-\eqref{eq:Zplus}.\\
    For an arbitrarily given $t \in \mathbb{R}$, let $j$ denote the unique integer for which $t \in [\theta_j, \theta_{j+1})$, and for any integer $k+1 \leq j$, where $s \in [\zeta_k, \zeta_{k+1}]$, we have
       \begin{align*}
           G(t, s) = Z_-(t, \theta_{k+1}) X(\theta_{k+1}, s)
       \end{align*}
       By utilizing \eqref{eq:Zminus}, we obtain:
       \begin{align*}
           ||G(t, s)||\leq||K_1e^{-\sigma_1(t-\theta_{k+1})}||||X(\theta_{k+1}, s)||
       \end{align*}
       Since both $\theta_{k+1}$ and $s$ belong to the same continuity interval, the term $\|X(\theta_{k+1}, s)\|$ remains bounded by $\bar{M}$. Thus,
       \begin{align*}
           ||G(t, s)|| &\leq K_1\bar{M}e^{-\sigma_1(t-\theta_{k+1})}\\
           &= K_1\bar{M}e^{-\sigma_1(t-s)}e^{-\sigma_1(s-\theta_{k+1})}\\
           &\leq K_1\bar{M}e^{\sigma_1\bar{\theta}}e^{-\sigma_1(t-s)}
       \end{align*}
       Likewise, for any integer $k \geq j$, where $s \in [\zeta_k, \zeta_{k+1}]$, we have
       \begin{align*}
           G(t, s) = Z_+(t, \theta_{k+1}) X(\theta_{k+1}, s)
       \end{align*}
       By utilizing \eqref{eq:Zplus}, we obtain:
       \begin{align*}
           ||G(t, s)||&\leq||K_2e^{\sigma_2(t-\theta_{k+1})}||||X(\theta_{k+1}, s)||\\
           &\leq K_2\bar{M}e^{\sigma_2(t-\theta_{k+1})}\\
           &= K_2\bar{M}e^{\sigma_2(t-s)}e^{\sigma_2(s-\theta_{k+1})}\\
           &\leq K_2\bar{M}e^{\sigma_2\bar{\theta}}e^{\sigma_2(t-s)}
       \end{align*}
       Lastly, if $s\in \widehat{[\zeta_j,t]}$, then
       \begin{align*}
           G(t,s)= X(t,s)
       \end{align*}
       Because $s$ and $t$ are in the same continuity interval,
       \begin{align*}
           ||G(t,s)|| \leq \bar{M}
       \end{align*}
       Choose $\sigma=\min \{\sigma_1,\sigma_2\}$ so that
       \begin{align*}
           ||G(t,s)|| &\leq \bar{M}e^{\sigma|t-s|}e^{-\sigma|t-s|}\\
           &\leq \bar{M}e^{\sigma\bar{\theta}}e^{-\sigma|t-s|} \qquad \text{for $s\in \widehat{[\zeta_j,t]}$}
       \end{align*}
       Then
       \begin{align*}
           ||G(t, s)||\leq 
           \begin{cases}
               K_1\bar{M}e^{\sigma_1\bar{\theta}}e^{-\sigma(t-s)}, \qquad s\leq \zeta_j\\
               K_2\bar{M}e^{\sigma_2\bar{\theta}}e^{\sigma(t-s)}, \qquad s\geq \zeta_j\\
               \bar{M}e^{\sigma\bar{\theta}}e^{-\sigma|t-s|}, \quad \qquad s\in \widehat{[\zeta_j,t]}
           \end{cases} 
       \end{align*}
       Further, choose $K=\max\{K_1 \bar{M} e^{\sigma_1 \bar{\theta}}, K_2 \bar{M} e^{\sigma_2 \bar{\theta}},\bar Me^{\sigma\bar{\theta}}\}$ to get the desired estimate:
       \begin{align*}
           ||G(t,s)|| \leq Ke^{-\sigma|t-s|}, \qquad\forall t,s \in \mathbb{R}
       \end{align*}
\end{proof}
The exponential estimates provided by the dichotomy are crucial for the boundedness of the function $z(t)$ in \eqref{eqnG}. To rigorously establish that $z(t)$ is a well-defined and bounded function on $\mathbb{R}$, we state and prove the following lemma.
\begin{lemma}\label{lemma3}
Suppose that matrices $A_0(t), A_1(t)$ and the function $f(t)$, $t\in\mathbb{R}$ are continuous and uniformly bounded on $\mathbb{R}$ and the conditions (C3), (C4) are  satisfied. If the system \eqref{eq:first} admits an exponential dichotomy, then the function $z(t)$ determined by \eqref{eqnG} is bounded on $\mathbb{R}$.
\end{lemma}
\begin{proof}
Assume that the function $f$ is bounded on the real axis by $\tilde{M}$, and consider   
\begin{align}
    z(t) &= \int\limits_{-\infty}^{\infty}G(t,s)f(s)ds =\int\limits_{-\infty}^tG(t,s)f(s)ds + \int\limits_t^{\infty}G(t,s)f(s)ds\label{z(t)}
\end{align}
Then using the estimate \eqref{||G(t,s)||} and the boundedness of $f$, we have 
\begin{align*}
    ||z(t)|| &\leq\left\|\hspace{0,1cm}\int\limits_{-\infty}^tG(t,s)f(s)ds\hspace{0,1cm}\right\|+\left\|\hspace{0,1cm}\int\limits_{t}^{\infty}G(t,s)f(s)ds\hspace{0,1cm}\right\| \\
    &\leq \int\limits_{-\infty}^t||G(t,s)||||f(s)||ds\hspace{0,1cm}+\int\limits_{t}^{\infty}||G(t,s)||||f(s)||ds\\
    &\leq \int\limits_{-\infty}^t K\tilde{M}e^{-\sigma(t-s)}ds+\int\limits_t^{\infty} K\tilde{M}e^{-\sigma(s-t)}ds\\
    &\leq \frac{K\tilde{M}}{\sigma} + \frac{K\tilde{M}}{\sigma} = \frac{2K\tilde{M}}{\sigma}
\end{align*}
\end{proof}
In hybrid systems, the interaction between continuous dynamics and discrete updates necessitates a rigorous check of the solution's regularity at the switching moments. To ensure that our construction maintains a well-defined trajectory across the entire domain, we establish the continuity of the function $z(t)$ which is represented by \eqref{eqnG} in the following lemma.
\begin{lemma}
Suppose that matrices $A_0(t), A_1(t)$ and the function $f(t)$, $t\in\mathbb{R}$ are continuous and uniformly bounded on $\mathbb{R}$ and the conditions (C3), (C4) are  satisfied. If the system \eqref{eq:first} admits an exponential dichotomy, then the function $z(t)$ determined by \eqref{eqnG} is continuous on $\mathbb{R}$.
\end{lemma}
\begin{proof}
It is obvious that $z(t)$ is continuous in any interval $(\theta_i,\theta_{i+1}),i\in \mathbb{Z}.$ We only need to check whether it is also continuous at points $\theta_i, i\in\mathbb{Z.}$\\ \indent Fix, $j\in \mathbb{Z},$ and evaluate
\begin{align*}
z(\theta_j+) 
&= \int_{-\infty}^{\infty} G(\theta_j+, s) f(s)\, ds = \int_{-\infty}^{\theta_j+} G(\theta_j+, s) f(s)\, ds 
+ \int_{\theta_j+}^{\infty} G(\theta_j+, s) f(s)\, ds \\
&= \sum_{k=-\infty}^j Z_-(\theta_j, \theta_k) 
\int_{\zeta_{k-1}}^{\zeta_k} X(\theta_k, s) f(s)\, ds 
+ \int_{\zeta_j}^{\theta_j} X(\theta_j) P X^{-1}(s) f(s)\, ds \\
&\quad - \sum_{k=j+1}^{\infty} Z_+(\theta_j, \theta_k) 
\int_{\zeta_{k-1}}^{\zeta_k} X(\theta_k, s) f(s)\, ds 
- \int_{\theta_j}^{\zeta_j} X(\theta_j) (I - P) X^{-1}(s) f(s)\, ds
\end{align*}
and
\begin{align*}
z(\theta_j-) 
&= \int_{-\infty}^{\infty} G(\theta_j-, s) f(s)\, ds = \int_{-\infty}^{\theta_j-} G(\theta_j-, s) f(s)\, ds 
+ \int_{\theta_j-}^{\infty} G(\theta_j-, s) f(s)\, ds \\
&= \sum_{k=-\infty}^{j-1} Z_-(\theta_j, \theta_k) 
\int_{\zeta_{k-1}}^{\zeta_k} X(\theta_k, s) f(s)\, ds 
+ \int_{\zeta_{j-1}}^{\theta_j} X(\theta_j) P X^{-1}(s) f(s)\, ds \\
&\quad - \sum_{k=j}^{\infty} Z_+(\theta_j, \theta_k) 
\int_{\zeta_{k-1}}^{\zeta_k} X(\theta_k, s) f(s)\, ds 
- \int_{\theta_j}^{\zeta_{j-1}} X(\theta_j) (I - P) X^{-1}(s) f(s)\, ds
\end{align*}
\indent Subtract from the first formula the second:
\begin{align*}
&z(\theta_j+)-z(\theta_j-)= \\
&\quad \sum_{k=-\infty}^j Z_-(\theta_j, \theta_k) 
\int_{\zeta_{k-1}}^{\zeta_k} X(\theta_k, s) f(s)\, ds -\sum_{k=-\infty}^{j-1} Z_-(\theta_j, \theta_k) 
\int_{\zeta_{k-1}}^{\zeta_k} X(\theta_k, s) f(s)\, ds\\
&\quad - \sum_{k=j+1}^{\infty} Z_+(\theta_j, \theta_k) 
\int_{\zeta_{k-1}}^{\zeta_k} X(\theta_k, s) f(s)\, ds +\sum_{k=j}^{\infty} Z_+(\theta_j, \theta_k) 
\int_{\zeta_{k-1}}^{\zeta_k} X(\theta_k, s) f(s)\, ds \\
&\quad +\int_{\zeta_j}^{\theta_j} X(\theta_j) P X^{-1}(s) f(s)\, ds-\int_{\zeta_{j-1}}^{\theta_j} X(\theta_j) P X^{-1}(s) f(s)\, ds\\
&\quad - \int_{\theta_j}^{\zeta_j} X(\theta_j) (I - P) X^{-1}(s) f(s)\, ds+\int_{\theta_j}^{\zeta_{j-1}} X(\theta_j) (I - P) X^{-1}(s) f(s)\, ds\\
&=Z_-(\theta_j,\theta_j)\int_{\zeta_{j-1}}^{\zeta_j}X(\theta_j,s)f(s)ds-Z_+(\theta_j,\theta_j)\int_{\zeta_{j-1}}^{\zeta_j}X(\theta_j,s)f(s)ds\\
&\quad - \int\limits_{\zeta_{j-1}}^{\zeta_j}X(\theta_j)PX^{-1}(s)f(s)ds-\int\limits_{\zeta_{j-1}}^{\zeta_j}X(\theta_j)(I-P)X^{-1}(s)f(s)ds\\
&=[Z_-(\theta_j,\theta_j)-Z_+(\theta_j,\theta_j)]\int_{\zeta_{j-1}}^{\zeta_j}X(\theta_j,s)f(s)ds-\int_{\zeta_{j-1}}^{\zeta_j}X(\theta_j,s)f(s)\\
&=\int_{\zeta_{j-1}}^{\zeta_j}X(\theta_j,s)f(s)ds-\int_{\zeta_{j-1}}^{\zeta_j}X(\theta_j,s)f(s)=0
\end{align*}
Hence $z(t)$ is continuous on $\mathbb{R}.$
\end{proof}
\begin{remark}
   Although the continuity of the function $z(t)$ represented by \eqref{eqnG} is explicitly demonstrated in the proof through detailed summations, it can also be understood directly from \eqref{z(t)}. In this representation, the first integral has an upper limit of $t$ and the second has a lower limit of $t$, ensuring that both components are continuous functions of $t$. Since the summation of continuous functions is itself continuous, it follows that the function $z(t)$ is continuous across the entire real line. We have provided the expanded summation-based proof primarily to offer a more transparent view of the interactions at the switching moments $\theta_i$ for all indices $i$.
\end{remark}
Having established the boundedness on the entire real axis and the continuity of $z(t)$ represented by \eqref{eqnG}, we now proceed to verify that this representation indeed satisfies the nonhomogeneous system \eqref{eq:second}. These properties are not merely auxiliary; they provide the necessary analytical foundation to ensure that $z(t)$ remains a well-defined solution across the entire real line. Consequently, the lemmas provided above serve as the basis for the following formal confirmation that this representation constitutes a solution to the differential equation under investigation.
\begin{lemma}
If the associated homogeneous system \eqref{eq:first} admits an exponential dichotomy, then the function $z(t)$ determined by \eqref{eqnG}
is the solution of equation \eqref{eq:second} in the sense of Definition \ref{solndef}. 
\end{lemma}
\begin{proof} Fix an arbitrary integer $j \in \mathbb{Z}$ and consider $t \in [\theta_j, \theta_{j+1})$. For any $t \in (\theta_j, \theta_{j+1})$, we compute the derivative $z'(t)$ as follows
\begin{align*}
    z'(t) &= \frac{d}{dt}\left[\int\limits_{-\infty}^{\infty}G(t,s)f(s)ds\right] = \frac{d}{dt}\left[\int\limits_{-\infty}^tG(t,s)f(s)ds + \int\limits_t^{\infty}G(t,s)f(s)ds\right] \\
    &= \frac{d}{dt}\left[\sum_{k=-\infty}^j Z_-(t, \theta_k) \int_{\zeta_{k-1}}^{\zeta_k} X(\theta_k, s) f(s)\, ds \right] + \frac{d}{dt}\left[\int_{\zeta_j}^{t} X(t) P X^{-1}(s) f(s)\, ds\right] \\
    &\quad - \frac{d}{dt}\left[\sum_{k=j+1}^{\infty} Z_+(t, \theta_k) \int_{\zeta_{k-1}}^{\zeta_k} X(\theta_k, s) f(s)\, ds\right] - \frac{d}{dt}\left[\int_{t}^{\zeta_j} X(t) (I - P) X^{-1}(s) f(s)\, ds\right] \\
    &= \sum\limits_{k=-\infty}^j\left[ \frac{d}{dt}Z_-(t,\theta_k) \int\limits_{\zeta_{k-1}}^{\zeta_k} X(\theta_k,s) f(s) \, ds\right] - \sum\limits_{k=j+1}^{\infty}\left[ \frac{d}{dt}Z_+(t,\theta_k) \int\limits_{\zeta_{k-1}}^{\zeta_k} X(\theta_k,s) f(s) \, ds\right] \\
    &\quad + \frac{d}{dt}\left[\int\limits_{\zeta_j}^t X(t,s) f(s) \, ds \right] \\
\intertext{Then, by using the definition of $Z_-(t,s)$ and $Z_+(t,s)$ as given in \eqref{EXPDICH}, we get}
    z'(t) &= \sum\limits_{k=-\infty}^j\left[ \frac{dZ(t)}{dt}PZ^{-1}(\theta_k) \int\limits_{\zeta_{k-1}}^{\zeta_k} X(\theta_k,s) f(s) \, ds\right] \\
    &\quad - \sum\limits_{k=j+1}^{\infty}\left[ \frac{dZ(t)}{dt}(I-P)Z^{-1}(\theta_k) \int\limits_{\zeta_{k-1}}^{\zeta_k} X(\theta_k,s) f(s) \, ds\right] + \int\limits_{\zeta_j}^t \frac{dX(t,s)}{dt} f(s) \, ds + X(t,t)f(t) \\
\intertext{Substituting $\frac{dZ(t)}{dt}=A_0(t)Z(t)+A_1(t)Z(\gamma(t))$ and $\frac{dX(t,s)}{dt}=A_0(t)X(t,s)$ in the previous equation, we have}
    z'(t) &= \sum\limits_{k=-\infty}^j\left\{[ A_0(t)Z(t)+A_1(t)Z(\gamma(t))]PZ^{-1}(\theta_k) \int\limits_{\zeta_{k-1}}^{\zeta_k} X(\theta_k,s) f(s) \, ds\right\} \\
    &\quad - \sum\limits_{k=j+1}^{\infty}\left\{[ A_0(t)Z(t)+A_1(t)Z(\gamma(t))](I-P)Z^{-1}(\theta_k) \int\limits_{\zeta_{k-1}}^{\zeta_k} X(\theta_k,s) f(s) \, ds\right\} \\
    &\quad + \int\limits_{\zeta_j}^t[ A_0(t)X(t,s) f(s)] \, ds + f(t) \\
\intertext{and observe that}
    z'(t) &= A_0(t)\left[\sum\limits_{k=-\infty}^j Z_-(t,\theta_k) \int\limits_{\zeta_{k-1}}^{\zeta_k} X(\theta_k,s) f(s) \, ds \right. \\
    &\quad \left. - \sum\limits_{k=j+1}^\infty Z_+(t,\theta_k) \int\limits_{\zeta_{k-1}}^{\zeta_k} X(\theta_k,s) f(s) \, ds + \int\limits_{\zeta_j}^t X(t,s) f(s) \, ds\right]\\
    &\quad + A_1(t)\left[\sum\limits_{k=-\infty}^j Z_-(\gamma(t),\theta_k) \int\limits_{\zeta_{k-1}}^{\zeta_k} X(\theta_k,s) f(s) \, ds \right. \\
    &\quad \left. - \sum\limits_{k=j+1}^\infty Z_+(\gamma(t),\theta_k) \int\limits_{\zeta_{k-1}}^{\zeta_k} X(\theta_k,s) f(s) \, ds\right] \\
    &\quad + A_1(t)\int\limits_{\zeta_j}^{\zeta_j} X(t,s) f(s) \, ds + f(t) \\
    &= A_0(t)\left[\int\limits_{-\infty}^{t}G(t,s)f(s)ds+\int\limits_{t}^{\infty}G(t,s)f(s)ds\right] \\
    &\quad + A_1(t)\left[\int\limits_{-\infty}^{\gamma(t)}G(\gamma(t),s)f(s)ds+\int\limits_{\gamma(t)}^{\infty}G(\gamma(t),s)f(s)ds\right] + f(t) \\
    &= A_0(t)\left[\int\limits_{-\infty}^{\infty}G(t,s)f(s)ds\right] + A_1(t)\left[\int\limits_{-\infty}^{\infty}G(\gamma(t),s)f(s)ds\right] + f(t)\\
    &= A_0(t)z(t) + A_1(t)z(\gamma(t)) + f(t).
\end{align*}
At $t = \theta_i$, according to Definition \ref{solndef}, we must verify that the equation holds for the right-hand derivative $z'(\theta_i+)$. Since the coefficients $A_0(t), A_1(t)$, the forcing term $f(t)$, and the piecewise constant argument $\gamma(t)$ are all right-continuous at each switching moment, their combination in the differential equation is well-defined. Consequently, the right-hand derivative $z'(\theta_i+)$ exists and is equal to the right-hand limit of the expression$$z'(\theta_i+) = \lim_{t \to \theta_i+} z'(t) = A_0(\theta_i)z(\theta_i) + A_1(\theta_i)z(\gamma(\theta_i)) + f(\theta_i)$$This confirms that $z(t)$ satisfies the equation at each point $\theta_i$ in the sense of the right-hand derivative, fulfilling the final requirement of the solution definition. 
Thus the lemma is proved.
\end{proof}
\begin{remark} \label{remark_uniqueness}
Having established the existence of a bounded solution, we emphasize that its uniqueness is a direct consequence of the structural properties of an exponential dichotomy. According to Coppel's framework, a system possesses an exponential dichotomy on $\mathbb{R}$ if and only if the corresponding homogeneous equation has no nontrivial solution bounded on $\mathbb{R}$ \cite{14}. This fundamental principle remains valid for our model, as the presence of a piecewise constant argument does not alter the underlying dichotomic structure. Since the homogeneous part of the system \eqref{eq:first} satisfies these requirements across the entire real line, the solution $z(t)$ represented in \eqref{eqnG} is the unique bounded almost periodic solution of equation \eqref{eq:second}.
\end{remark}
\section{Existence of the Almost Periodic Solution}
\indent \indent In EPCAG systems, the almost periodicity of the solution depends on more than just the continuous functions; it is also determined by the distribution of the switching moments. For the solution to remain almost periodic, the shifts in the continuous functions must be synchronized with the shifts in the discrete sequences of the arguments.\\
\indent The following assertion is a specific one.
\begin{lemma} \cite{1}
Assume that the functions $\phi_j(t), j=1,2,...,k$ are almost periodic in $t$, $\{\zeta_i^j\}_i$, $j\in \mathbb{Z}$, are uniformly almost periodic and $\inf\limits_\mathbb{Z}\zeta_i^1>0$. Then, for arbitrary $\eta>0$, $0<\nu<\eta$, there exist a relatively dense set of real numbers $\Omega$ and integers $Q$, such that for $\omega\in \Omega$, $q\in Q$, it is true that 
      \begin{enumerate}
          \item $||\phi_j(t+\omega)-\phi_j(t)||<\eta$, $j=1,2,...,k$,\quad $t\in \mathbb{R}$;
          \item $|\theta_i^q-\omega| < \nu$, \quad $i\in \mathbb{Z}$;
          \item $|\zeta_i^q-\omega| < \nu$, \quad $i\in \mathbb{Z}$.
      \end{enumerate} \label{lemma4}
 \end{lemma}
To establish our main results, we require a diagonal almost periodicity estimate for the Green's function. To this end, we first define the positive constants $\tilde{\sigma} = \frac{\sigma}{2}$ and $\tilde{K} = \max \{ \frac{2K^2\bar{M} e^{3\sigma\bar{\theta}}}{\sigma}, \frac{2K^2e^{2\sigma\bar{\theta}}}{\sigma}, \frac{2K\bar{M}e^{3\sigma\bar{\theta}}}{\sigma}, \bar{M}^2\bar{\theta}e^{\frac{\sigma}{2}\bar{\theta}}\}.$ \\
\indent In view of Lemma \ref{lemma4}, we formulate the following auxiliary lemma to provide the desired diagonal almost periodicity estimate for all $t,s \in \mathbb{R}$.

\begin{lemma} \label{DiagAP}
    Let $\omega$ be a common $\eta$-almost period of matrices $A_0(t)$ and $A_1(t)$. Then the Green's function satisfies 
    \begin{align}
        ||G(t+\omega,s+\omega)-G(t,s)|| \leq \tilde{K}\eta e^{-\tilde{\sigma}|t-s|}. \label{G-G}
    \end{align}
\end{lemma}
\begin{lemma} \label{DiagAP}
    Let $\omega$ be a common $\eta$-almost period of matrices $A_0(t)$ and $A_1(t)$. Then the Green's function satisfies the following diagonal almost periodicity estimated for all $t,s \in \mathbb{R}$.
    \begin{align}
        ||G(t+\omega,s+\omega)-G(t,s)|| \leq \tilde{K}\eta e^{-\tilde{\sigma}|t-s|}\label{G-G}
    \end{align}
\end{lemma}
\begin{proof}
    Take $q \in \mathbb{Z}$ that satisfies the Lemma \ref{lemma4} with matrix-functions $A_0(t)$ and $A_1(t)$. Then, there exist numbers $\nu_1$ and $\nu_2$ with $|\nu_k| < \eta$, $k=1,2$ such that $\zeta_{k+q}=\zeta_k+\omega+\nu_1$ and $\zeta_{k+q+1}=\zeta_{k+1}+\omega+\nu_2$.\\
For the estimation of the diagonal almost periodicity, the following three cases are considered according to the position of $s$, accounting for the different dynamical behaviors of the Green's function.  \vspace{0,2cm} \\
Assume that $t\in [\theta_j,\theta_{j+1})$ for some index $j$. \vspace{0,2cm}\\
\indent \textit{CASE I:} $s\in[\zeta_i,\zeta_{i+1}]$, $i\in \mathbb{Z}$, $i+1\leq j.$\vspace{0,2cm}\\
Set $ W_-(t,s) = G(t+\omega,s+\omega)-G(t,s)$. Then
\begin{align*}
    W_-(t,s) 
    &= \begin{cases}
        Z_-(t+\omega,\theta_{i+q})X(\theta_{i+q},s+\omega)
        \\- Z_-(t,\theta_{i+1})X(\theta_{i+1},s) 
        &\quad \text{if } s \in [\zeta_i,\zeta_i+\nu_1] \\[6pt]
        Z_-(t+\omega,\theta_{i+q+1})X(\theta_{i+q+1},s+\omega)
        \\- Z_-(t,\theta_{i+1})X(\theta_{i+1},s) 
        &\quad \text{if } s \in [\zeta_i+\nu_1,\zeta_{i+1}+\nu_2]\\[6pt]
        Z_-(t+\omega,\theta_{i+q+2})X(\theta_{i+q+2},s+\omega)
        \\- Z_-(t,\theta_{i+1})X(\theta_{i+1},s) 
        &\quad \text{if } s \in [\zeta_{i+1}+\nu_2,\zeta_{i+1}]
    \end{cases}
\end{align*}
For each case, we have
\begin{align*}
    \frac{\partial W_-(t,s)}{\partial t}&=A_0(t)W_-(t,s)+A_1(t)W_-(\gamma(t),s)\\
    &\quad+[A_0(t+\omega)-A_0(t)]G(t+\omega,s+\omega)\\
    &\quad+[A_1(t+\omega)-A_1(t)]G(\gamma(t+\omega),s+\omega)
\end{align*}
Since $W_-(s,s)=0$, we have 
\begin{align*}
    W_-(t,s)&=Z_-(t,s)\int\limits_s^{\zeta_i}X(s,u)\mathcal{D}(u,s)du+\sum\limits_{k=i}^{k=j-1}Z_-(t,\theta_{k+1})\int\limits_{\zeta_k}^{\zeta_{k+1}}X(\theta_{k+1},u)\mathcal{D}(u,s)du\\ 
    &\quad +\int\limits_{\zeta_j}^tX(t,u)\mathcal{D}(u,s)du\\ 
    &=Z_-(t,s)\int\limits_s^{\zeta_i}X(s,u)\mathcal{D}(u,s)du+\sum\limits_{k=i}^{k=j-1}\int\limits_{\zeta_k}^{\zeta_{k+1}}G(t,u)\mathcal{D}(u,s)du\\ 
    &\quad +\int\limits_{\zeta_j}^tX(t,u)\mathcal{D}(u,s)du
\end{align*}
where $\mathcal{D}(u,s)=[A_0(u+\omega)-A_0(u)]G(u+\omega,s+\omega)+[A_1(u+\omega)-A_1(u)]G(\gamma(u+\omega),s+\omega)$\\  
Observe that
\begin{align*}
    ||\mathcal{D}(u,s)|| &\leq \eta Ke^{-\sigma(u-s)}+\eta Ke^{-\sigma(\gamma(u+\omega)-(s+\omega))}\leq \eta Ke^{-\sigma(u-s)}+\eta Ke^{2\sigma\bar{\theta}}e^{-\sigma(u-s)}\\
    &\leq 2\eta Ke^{2\sigma\bar{\theta}}e^{-\sigma(u-s)}
\end{align*}
Then we have that
\begin{align*}
    ||W_-(t,s)|| &\leq \int\limits_s^{\zeta_i}||Z_-(t,s)||||X(s,u)||||\mathcal{D}(u,s)||du + \sum\limits_{k=i}^{j-1}\int\limits_{\zeta_k}^{\zeta_{k+1}}||G(t,u)||||\mathcal{D}(u,s)||du\\
    &+\int\limits_{\zeta_j}^{t}||X(t,u)||||\mathcal{D}(u,s)||du
    \\
    &\leq \int\limits_s^{\zeta_i} K_1e^{-\sigma_1(t-s)}\bar{M}2\eta Ke^{2\sigma\bar{\theta}}e^{-\sigma(u-s)}du+\sum\limits_{k=i}^{j-1}\int\limits_{\zeta_k}^{\zeta_{k+1}}Ke^{-\sigma(t-u)}2\eta Ke^{2\sigma\bar{\theta}}e^{-\sigma(u-s)}\\
    &+ \int\limits_{\zeta_j}^{t} \bar{M}2\eta Ke^{2\sigma\bar{\theta}}e^{-\sigma(u-s)}du
\end{align*}
Clearly,
\begin{align*}
    \int\limits_s^{\zeta_i} K_1e^{-\sigma_1(t-s)}\bar{M}2\eta Ke^{2\sigma\bar{\theta}}e^{-\sigma(u-s)}du &\leq \int\limits_s^{\zeta_i} Ke^{-\sigma(t-s)}\bar{M}2\eta Ke^{2\sigma\bar{\theta}}e^{-\sigma(u-s)}du \\
    & \leq \int\limits_s^{\zeta_i} K^2\bar{M}e^{3\sigma\bar{\theta}}2\eta e^{-\sigma(t-s)} du
\end{align*}
by the choice of $K$ and $\sigma$, also
\begin{align*}
    \int\limits_{\zeta_k}^{\zeta_{k+1}}Ke^{-\sigma(t-u)}2\eta Ke^{2\sigma\bar{\theta}}e^{-\sigma(u-s)} &\leq \int\limits_{\zeta_k}^{\zeta_{k+1}}K^2e^{2\sigma\bar{\theta}}2\eta e^{-\sigma(t-s)}
\end{align*}
and
\begin{align*}
    \int\limits_{\zeta_j}^{t} \bar{M}2\eta Ke^{2\sigma\bar{\theta}}e^{-\sigma(u-s)}du &\leq \int\limits_{\zeta_j}^{t} K\bar{M}e^{2\sigma\bar{\theta}}2\eta e^{-\sigma(u-s)}e^{-\sigma(t-u)}e^{\sigma(t-u)}du \\
    &\leq \int\limits_{\zeta_j}^{t} K\bar{M}e^{3\sigma\bar{\theta}}2\eta e^{-\sigma(t-s)}du 
\end{align*}
Now choose $\hat{K}=\max\{2K^2\bar{M} e^{3\sigma\bar{\theta}},2K^2e^{2\sigma\bar{\theta}},2K\bar{M}e^{3\sigma\bar{\theta}}\}$, so that\\
\begin{align*}
    ||W_-(t,s)|| &\leq \int\limits_s^t\hat{K}\eta e^{-\sigma(t-s)}=\hat{K}\eta e^{-\sigma(t-s)}(t-s) \leq \frac{\hat{K}}{\sigma}\eta e^{\frac{-\sigma}{2}(t-s)}
\end{align*}
\indent \textit{CASE II:} $s\in[\zeta_i,\zeta_{i+1}]$, $i\in \mathbb{Z}$, $i\geq j.$\vspace{0,2cm}\\
By applying analogous arguments to $W_+(t,s)=G(t+\omega,s+\omega)-G(t,s)$, the following estimate is obtained.
\begin{align*}
    ||G(t+\omega,s+\omega)-G(t,s)|| \leq \frac{\hat{K}}{\sigma}\eta e^{\frac{-\sigma}{2}(s-t)}
\end{align*}
\indent \textit{CASE III:} $s\in \widehat{[\zeta_j,t]}$.\vspace{0,2cm}\\
Set $W(t,s)=G(t+\omega,s+\omega)-G(t,s)=X(t+\omega,s+\omega)-X(t,s)$ and get
\begin{align*}
    \frac{\partial W(t,s)}{\partial t}=A_0(t)W(t,s)+[A_0(t+\omega)-A_0(t)]X(t+\omega,s+\omega)
\end{align*}
since $W(s,s)=0$, we have
\begin{align*}
    W(t,s)=\int\limits_s^tX(t,u)[A_0(u+\omega)-A_0(u)]X(u+\omega,s+\omega)du
\end{align*}
it is easy to see that
\begin{align*}
    ||W(t,s)|| \leq \int\limits_s^t\bar{M}^2\eta du &\leq \bar{M}^2\bar{\theta}\eta \\
    &\leq \bar{M}^2\bar{\theta}\eta e^{\frac{-\sigma}{2}|t-s|}e^{\frac{\sigma}{2}|t-s|}\\
    &\leq \bar{M}^2\bar{\theta}e^{\frac{\sigma}{2}\bar{\theta}}\eta e^{\frac{-\sigma}{2}|t-s|}
\end{align*}\\
Finally, choose $\tilde{K}=\max\{\frac{\hat{K}}{\sigma},\bar{M}^2\bar{\theta}e^{\frac{\sigma}{2}\bar{\theta}}\}$, and $\tilde{\sigma}:= \frac{\sigma}{2}$ to get the desired diagonal almost periodicity estimate. $$||G(t+\omega,s+\omega)-G(t,s)|| \leq \tilde{K}\eta e^{-\tilde{\sigma}|t-s|}, \quad\forall t,s \in \mathbb{R}. $$
\end{proof}
\begin{theorem} \label{mainthm}
Suppose that the conditions (C1)-(C4) hold. If the system \eqref{eq:first} admits exponential dichotomy, then equation \eqref{eq:second} admits a unique, continuous, and bounded almost periodic solution.
\end{theorem}
\begin{proof}
Consider the formula \eqref{z(t)}
\begin{align*}
    z(t)&=\int\limits_{-\infty}^{\infty}G(t,s)f(s)ds=\int\limits_{-\infty}^tG(t,s)f(s)ds+\int\limits_t^{\infty}G(t,s)f(s)ds
\end{align*}
\indent The existence, boundedness, and continuity of the solution $z(t)$ have already been established in Section 3. Furthermore, its uniqueness was discussed and confirmed in the preceding Remark \ref{remark_uniqueness} based on the structural properties of exponential dichotomy. Therefore, it remains only to verify that this unique bounded solution is indeed almost periodic.\\
\indent Let us check that the solution is almost periodic. For a given $\epsilon > 0$, there exists a sufficiently small $\eta = \eta(\epsilon) > 0$ that satisfies the requirements of Lemma \ref{lemma4} for the functions $A_0(t), A_1(t),$ and $f(t)$.\\
\indent This ensures the existence of a relatively dense set $J(\eta)$ of common $\epsilon$-almost periods $\omega$ for the functions $A_0(t), A_1(t),$ and $f(t)$. Consequently, for any $\omega \in J(\eta)$ and for $t \in [\theta_j, \theta_{j+1})$ for all indices $j$, we have
\begin{align*}
    z(t+\omega)-z(t)=&\int\limits_{-\infty}^{\infty}G(t+\omega,s)f(s)ds-\int\limits_{-\infty}^{\infty}G(t,s)f(s)ds\\
    =&\int\limits_{-\infty}^{t+\omega}G(t+\omega,s)f(s)ds+\int\limits_{t+\omega}^{\infty}G(t+\omega,s)f(s)ds\\
    &-\int\limits_{-\infty}^tG(t,s)f(s)ds-\int\limits_t^{\infty}G(t,s)f(s)ds\\ 
    =&\int\limits_{-\infty}^{t}G(t+\omega,s+\omega)f(s+\omega)ds+\int\limits_{t}^{\infty}G(t+\omega,s+\omega)f(s+\omega)ds\\
    &-\int\limits_{-\infty}^tG(t,s)f(s)ds-\int\limits_t^{\infty}G(t,s)f(s)ds\\ 
    =&\int\limits_{-\infty}^{t}[G(t+\omega,s+\omega)f(s+\omega)-G(t,s)f(s)]ds\\ 
    &+\int\limits_{t}^{\infty}[G(t+\omega,s+\omega)f(s+\omega)-G(t,s)f(s)]ds
\end{align*}
\indent To apply the diagonal almost periodicity estimation established in Lemma \ref{DiagAP}, we need the following transformations:
\begin{align*}
    &\int\limits_{-\infty}^{t}[G(t+\omega,s+\omega)f(s+\omega)-G(t,s)f(s)]ds=\\
    &\int\limits_{-\infty}^{t}\{[G(t+\omega,s+\omega)-G(t,s)]f(s+\omega)+G(t,s)[f(s+\omega)-f(s)]\}ds
\end{align*}
\indent and,
\begin{align*}
   & \int\limits_{t}^{\infty}[G(t+\omega,s+\omega)f(s+\omega)-G(t,s)f(s)]ds=\\
    &\int\limits_{t}^{\infty}\{[G(t+\omega,s+\omega)-G(t,s)]f(s+\omega)+G(t,s)[f(s+\omega)-f(s)]\}ds
\end{align*}
\indent Now, observe that
\begin{align*}
    ||G(t+\omega,s+\omega)f(s+\omega)-G(t,s)f(s)|| &\leq ||G(t+\omega,s+\omega)-G(t,s)||||f(s+\omega)||\\
    &+||G(t,s)||||f(s+\omega)-f(s)||
\end{align*}
\indent By applying the estimations \eqref{||G(t,s)||}, \eqref{G-G}, and recalling that the almost periodicity of $f(t)$ under condition (C1) implies its boundedness on the entire real line (with $\|f(t)\| \leq \tilde{M}$), we obtain,
\begin{align*}
    ||G(t+\omega,s+\omega)f(s+\omega)-G(t,s)f(s)|| &\leq \tilde{K}\tilde{M}\eta e^{-\tilde{\sigma}|t-s|} + K\eta e^{-\sigma|t-s|}
\end{align*}
\indent For convenience, choose $\bar{K} = \max \{ \tilde{K}\tilde{M}, K \}$ and $\bar{\sigma} = \min \{ \sigma, \tilde{\sigma} \}$. Then, we obtain
\begin{align*}
    ||G(t+\omega,s+\omega)f(s+\omega)-G(t,s)f(s)|| \leq 2\bar{K}\eta e^{-\bar{\sigma}|t-s|} 
\end{align*}
\indent Consequently, combining this estimation leads us
\begin{align*}
    ||z(t+\omega)-z(t)|| &\leq  \int\limits_{-\infty}^{t}||[G(t+\omega,s+\omega)f(s+\omega)-G(t,s)f(s)]||ds\\ 
    &+\int\limits_{t}^{\infty}||[G(t+\omega,s+\omega)f(s+\omega)-G(t,s)f(s)]||ds \\
    &\leq \int\limits_{-\infty}^{t}2\bar{K}\eta e^{-\bar{\sigma}(t-s)}ds+\int\limits_{t}^{\infty}2\bar{K}\eta e^{-\bar{\sigma}(s-t)}ds\\
    &\leq \frac{2\bar{K}\eta}{\bar{\sigma}} + \frac{2\bar{K}\eta}{\bar{\sigma}} = \frac{4\bar{K}}{\bar{\sigma}} \eta
\end{align*}
\indent To satisfy the definition of almost periodicity for the fixed $\epsilon > 0$, we choose the parameter $\eta = \eta(\epsilon)$ such that
\begin{equation*}
\eta < \frac{\bar{\sigma}}{4\bar{K}}\epsilon.
\end{equation*}
\indent Substituting this choice into our estimation, we obtain:
\begin{equation*}
||z(t+\omega)-z(t)|| < \frac{4\bar{K}}{\bar{\sigma}}\eta < \frac{4\bar{K}}{\bar{\sigma}}.\frac{\bar{\sigma}}{4\bar{K}}\epsilon = \epsilon.
\end{equation*}
\indent This inequality holds for any $\omega \in J(\eta) = J(\eta(\epsilon)) = \Omega(\epsilon)$, confirming that $z(t)$ possesses a relatively dense set of $\epsilon$-almost periods for every $\epsilon > 0$. Therefore, we conclude the almost periodicity of the solution $z(t)$.
\end{proof}
It is worth mentioning that the generalized piecewise constant argument $\gamma(t)$ encompasses both the retarded argument $\beta(t)$ and the advanced argument $\alpha(t)$ as important special cases. Specifically, by choosing the sequence of points $\zeta_i = \theta_i$, the argument $\gamma(t)$ reduces to the retarded case $\beta(t) = \theta_i$ for $t \in [\theta_i, \theta_{i+1})$ for all integer indices $i$. Also, setting $\zeta_i = \theta_{i+1}$ yields the advanced argument $\alpha(t) = \theta_{i+1}$.\\
\indent In a similar vein to the general case, the existence of a unique solution for the retarded and advanced systems depends on the construction of the fundamental matrix solution. For any fixed $(t_0, z_0) \in \mathbb{R} \times \mathbb{R}^n$, the solution is uniquely determined by a single initial moment, similar to ordinary differential equations as previously discussed. The fundamental matrix solution is determined by the transition relations at the switching moments, which are characterized by the matrices $M_i(t)$, for all indices $i$ under the condition (C4). For these specific configurations, we explain the fundamental matrix solutions explicitly as follows.\\
\indent To explore these dynamics under the same assumptions for the matrix-functions $A_0(t)$ and $A_1(t)$ as provided in Section 1, we first define the corresponding retarded system:
\begin{align}
z'(t) &= A_0(t)z(t) + A_1(t)z(\beta(t)), \label{eq:retarded} 
\end{align}
\indent Assume that $(t_0,z_0)$ is fixed, and there exists an index $i$ such that $\theta_i \leq t_0 < \theta_{i+1}$.  The solution of the system \eqref{eq:retarded} satisfies, on the interval $[\theta_i,\theta_{i+1})$ $i\in\mathbb{Z}$, the functional differential equation
\begin{equation*}
    z'(t)=A_0(t)z(t)+A_1(t)z(\theta_i).
\end{equation*}
\indent Assume that $t_0 \neq \theta_i$. Then, formally, we need the pair of initial points $(t_0,z_0)$ and $(\theta_i,z(\theta_i))$ to proceed with the solution. Indeed, since $z_0=M_i(t_0)z(\theta_i)$ and the matrix $M_i(t_0)$ is nonsingular, we can say that the initial condition $z(t_0) = z_0$ is sufficient to define the solution.\\
\indent Hence, for a fixed $t_0\in \mathbb{R}$ there exists a fundamental matrix of solutions of \eqref{eq:retarded},\\ $Z(t)=Z(t,t_0)$, $Z(t_0,t_0)=I$ such that 
\begin{equation*}
    \frac{dZ}{dt}=A_0(t)Z(t)+A_1(t)Z(\beta(t))
\end{equation*}
Let us construct $Z(t)$. Define the matrix for increasing t.\vspace{0,1cm}\\
\indent We have $Z(\theta_i)=M_i^{-1}(t_0)I=M_i^{-1}(t_0).$ Hence, on the interval $[\theta_i,\theta_{i+1}]$,$$Z(t,t_0)=M_i(t)M_i^{-1}(t_0).$$Then $Z(\theta_{i+1})=M_i(\theta_{i+1})M_i^{-1}(t_0)$ which implies that$$Z(t,t_0)=M_{i+1}(t)Z(\theta_{i+1})=M_{i+1}(t)M_i(\theta_{i+1})M_i^{-1}(t_0)$$if $t \in [\theta_{i+1},\theta_{i+2}]$. Then one can continue by induction to obtain
\begin{equation*}
    Z(t)=M_l(t)\left[ \prod\limits_{k=l-1}^{i}M_k(\theta_{k+1}) \right]M_i^{-1}(t_0), \label{eqn7}
\end{equation*}
if $t\in[\theta_l,\theta_{l+1}]$, for arbitrary $l>i$.\vspace{0,2cm} \\
\indent Similarly, for decreasing $t$, we have the following discussion. If $t \in [\theta_{i-1},\theta_i]$, then$$ Z(t,t_0)=M_{i-1}(t)Z(\theta_{i-1}). $$We can find $Z(\theta_{i-1})$ as follows:$$ Z(\theta_i)= M_i(\theta_i)M_i^{-1}(t_0) \quad \text{and also} \quad Z(\theta_i)=M_{i-1}(\theta_i)Z(\theta_{i-1})$$ by continuity. Thus,$$ Z(\theta_{i-1})=M_{i-1}^{-1}(\theta_i)Z(\theta_i)=M_{i-1}^{-1}(\theta_i)M_i(\theta_i)M_i^{-1}(t_0). $$Hence, on the interval $[\theta_{i-1},\theta_i]$,$$ Z(t,t_0)=M_{i-1}(t)M_{i-1}^{-1}(\theta_i)M_i(\theta_i)M_i^{-1}(t_0). $$Then, by induction, we have
\begin{equation*}
    Z(t)=M_j(t)\left[ \prod\limits_{k=j}^{i-1}M_k^{-1}(\theta_{k+1})M_{k+1}(\theta_{k+1}) \right]M_i^{-1}(t_0).
\end{equation*}
if $t\in[\theta_j,\theta_{j+1}]$, for arbitrary $j<i$.\vspace{0,2cm} \\
\indent One can easily see that 
\begin{equation*}
Z(t,s)=Z(t)Z^{-1}(s), \quad t,s \in \mathbb{R}
\end{equation*}
   and a solution $z(t),z(t_0)=z_0, (t_0.z_0) \in \mathbb{R}\times\mathbb{R}^n$ of \eqref{eq:retarded} is equal to 
   \begin{equation*}
       z(t)=Z(t,t_0)z_0 \quad t\in \mathbb{R.} \label{eqn10}
   \end{equation*}
\indent Also, define the corresponding advanced system: \\
\begin{align}
    z'(t) &= A_0(t)z(t) + A_1(t)z(\alpha(t)). \label{eq:advanced}
\end{align}
and the solution of the system \eqref{eq:advanced} satisfies, on the given interval $[\theta_i,\theta_{i+1})$, the functional differential equation

\begin{equation*}
    z'(t)=A_0(t)z(t)+A_1(t)z(\theta_{i+1}).
\end{equation*}
and again by assuming that $t_0 \neq \theta_{i+1}$, we need the pair of initial points $(t_0,z_0)$ and $(\theta_i,z(\theta_{i+1}))$ to proceed with the solution. Indeed, since $z_0=M_i(t_0)z(\theta_{i+1})$ and the matrix $M_i(t_0)$ is nonsingular, we can say that the initial condition $z(t_0) = z_0$ is sufficient to define the solution.\\
\indent Hence, for a fixed $t_0\in \mathbb{R}$ there exists a fundamental matrix of solutions of \eqref{eq:advanced},\\ $Z(t)=Z(t,t_0)$, $Z(t_0,t_0)=I$ such that 
\begin{equation*}
    \frac{dZ}{dt}=A_0(t)Z(t)+A_1(t)Z(\alpha(t)).
\end{equation*}
Let us construct $Z(t)$ for increasing t first.\vspace{0,1cm}\\
\indent We have $Z(\theta_{i+1})=M_i^{-1}(t_0)I=M_i^{-1}(t_0)$. Thus on the interval $[\theta_i,\theta_{i+1}]$, $$Z(t,t_0)=M_i(t)M_i^{-1}(t_0).$$ If $t\in [\theta_{i+1},\theta_{i+2}]$,$$ Z(t,t_0)=M_{i+1}(t)Z(\theta_{i+2}). $$To find $Z(\theta_{i+2})$, use $$Z(\theta_{i+1})=M_i(\theta_{i+1})M_i^{-1}(t_0) \quad \text{and} \quad Z(\theta_{i+1})=M_{i+1}(\theta_{i+1})Z(\theta_{i+2})$$ because of the continuity. Observe that$$ Z(\theta_{i+2})=M_{i+1}^{-1}(\theta_{i+1})Z(\theta_{i+1})=M_{i+1}^{-1}(\theta_{i+1})M_i(\theta_{i+1})M_i^{-1}(t_0). $$Hence,$$ Z(t,t_0)=M_{i+1}(t)M_{i+1}^{-1}(\theta_{i+1})M_i(\theta_{i+1})M_i^{-1}(t_0), $$ Then, by induction, one can find
\begin{equation*}
    Z(t)=M_l(t)\left[ \prod\limits_{k=l}^{i+1}M_k^{-1}(\theta_{k})M_{k-1}(\theta_{k}) \right]M_i^{-1}(t_0), \label{eqn7}
\end{equation*}
if $t\in[\theta_l,\theta_{l+1})$, for arbitrary $l>i$.\vspace{0,2cm} \\
\indent For decreasing $t$, we have the following matrix on the interval $[\theta_{i-1},\theta_i]$, $$ Z(t,t_0)=M_{i-1}(t)Z(\theta_i) $$ and $Z(\theta_i)=M_i(\theta_i)M_i^{-1}(t_0)$. Thus$$ Z(t,t_0)=M_{i-1}(t)M_i(\theta_i)M_i^{-1}(t_0).$$ Then, by induction, we get
\begin{equation*}
    Z(t)=M_j(t)\left[ \prod\limits_{k=j+1}^{i}M_{k}(\theta_{k}) \right]M_i^{-1}(t_0).
\end{equation*}
if $t\in[\theta_j,\theta_{j+1}]$, for arbitrary $j<i$.\vspace{0,2cm} \\
\indent One can easily see that 
\begin{equation*}
Z(t,s)=Z(t)Z^{-1}(s), \quad t,s \in \mathbb{R}
\end{equation*}
   and a solution $z(t),z(t_0)=z_0, (t_0,z_0) \in \mathbb{R}\times\mathbb{R}^n$ of \eqref{eq:advanced} is equal to 
   \begin{equation*}
       z(t)=Z(t,t_0)z_0 \quad t\in \mathbb{R.} \label{eqn10}
   \end{equation*}
\indent Therefore, the results established in Theorem \ref{mainthm} naturally extend to these specific equations. Based on this observation, we state the following corollaries.
\begin{corollary}
    Consider the following system:
\begin{align}
        z'(t) &= A_0(t)z(t) + A_1(t)z(\beta(t)) + f(t) \label{betanon}
\end{align}
with the associated system \eqref{eq:retarded} which is exponential dichotomous. Also suppose that the conditions (C1)-(C4) hold. Then the equation \eqref{betanon} admits a unique, continuous, and bounded almost periodic solution.
\end{corollary}
\begin{proof}
    The result is a direct consequence of Theorem \ref{mainthm} by setting $\zeta_i = \theta_{i}$, which yields $\gamma(t)=\beta(t)$
\end{proof}
\begin{corollary}
    Consider the following system:
\begin{align}
    z'(t) &= A_0(t)z(t) + A_1(t)z(\alpha(t)) + f(t) \label{alphanon}
\end{align}
with the associated system \eqref{eq:advanced} which is exponential dichotomous. Also suppose that the conditions (C1)-(C4) hold. Then the equation \eqref{alphanon} admits a unique, continuous, and bounded almost periodic solution.
\end{corollary}
\begin{proof}
    The result is a direct consequence of Theorem \ref{mainthm} by setting $\zeta_i = \theta_{i+1}$, which yields $\gamma(t)=\alpha(t)$
\end{proof}
The theoretical developments discussed so far provide the necessary framework to address the conceptual limitations of the model presented in \cite{24}. The approach in \cite{24} relies on a scalar equation to claim exponential dichotomic behavior. However, in a one-dimensional framework, an exponential dichotomy trivially reduces to basic exponential decay governed by the sign of the coefficient. Such an oversimplification overlooks the genuine essence of the phenomenon, which strictly requires the simultaneous existence of stable and unstable directions.\\
\indent Besides, in the literature, multi-variable systems with generalized piecewise constant arguments have already been explored within this context. Most notably, the foundational work in \cite{1}  provided a two-dimensional example specifically to illustrate uniform exponential stability, successfully establishing exact parameter bounds for this purely stable dynamic.\\
\indent While that study provided a rigorous example for the stable aspect of the dichotomy, our objective is to move beyond the trivial scalar case in \cite{24} and explicitly demonstrate a genuine saddle-type exponential dichotomy. In the following example, this behavior is quantitatively verified through the exact calculation of multipliers. By incorporating an almost periodic forcing term and a structured switching sequence $\theta_i$, we illustrate the simultaneous presence of stable and unstable subspaces, providing a comprehensive visualization of the exponential dichotomic separation in the solution space. \vspace{0.2cm}

\noindent \textbf{Example} Define the following equation:
\begin{equation}\label{eqnNonHom}
    z'(t)= \begin{pmatrix}
        0 & 1  \\ 0 & 0
    \end{pmatrix}z(t) + \begin{pmatrix}
        0 & 0 \\ q & 0
    \end{pmatrix}z(\gamma(t)) + \begin{pmatrix}
        \cos(2\pi t) \\ \sin(2t)
    \end{pmatrix} 
\end{equation} 
where \( q \in \mathbb{R} \), and the switching moments \( \theta_i \) are defined by the alternating sequence
\begin{equation*}
   \theta_i = 
   \begin{cases}
    i, & \text{if } i = 2k, \\
    i - \frac{1}{2}, & \text{if } i = 2k + 1,
\end{cases} \quad k \in \mathbb{Z},
\end{equation*}
Furthermore, the deviating moments are chosen as \( \zeta_i = \theta_{i+1} \) for all \( i \in \mathbb{Z} \).\\
\indent To analyze the discrete nature of the piecewise constant arguments, we incorporate the $\omega-p$ property, which provides the formal basis for establishing periodicity within the hybrid framework. This property requires the sequences of switching moments $\{\theta_i\}$ and piecewise constant arguments $\{\zeta_i\}$ to satisfy the relations $\theta_{i+p} = \theta_i + \omega$ and $\zeta_{i+p} = \zeta_i + \omega$ for some $\omega \in \mathbb{R}$ and a positive integer $p$. Specifically, the system under consideration is characterized by an alternating sequence, and this arrangement naturally satisfies the $\omega-p$ property with $\omega = 2$ and $p = 2$, thereby facilitating the calculation of the monodromy matrix and the subsequent determination of the multipliers.\\
 \indent Denote by $Q$ the product $\prod\limits_{k=1}^pG_k$, where matrices $G_k$ are equal to $M_k^{-1}(\theta_k)M_{k-1}(\theta_k)$, $k\in \mathbb{Z}$. We call the matrix $Q$, the monodromy matrix, and eigenvalues of the matrix, $\rho_j$, $j=1,2,...,n$, the multipliers.\\
 \indent The stability characteristics and the existence of an exponential dichotomy for system \eqref{eqnNonHom} are governed by the properties of its associated homogeneous system:
 \begin{equation}\label{eqnHom}
    z'(t)= \begin{pmatrix}
        0 & 1  \\ 0 & 0
    \end{pmatrix}z(t) + \begin{pmatrix}
        0 & 0 \\ q & 0
    \end{pmatrix}z(\gamma(t))
\end{equation} 
 \indent Specifically, the monodromy matrix for the equation \eqref{eqnHom} is:
\begin{align*}
    M_2^{-1}(\theta_2)M_1(\theta_{2})M_1^{-1}(\theta_1)M_0(\theta_{1})=\begin{pmatrix}
    \frac{16(3q+4)}{(q-8)(9q-8)} & \frac{4(9q+32)}{(q-8)(9q-8)} \\ \frac{4q(3q+32)}{(q-8)(9q-8)} & \frac{48q+(q+8)(9q+8)}{(q-8)(9q-8)}
\end{pmatrix}
\end{align*}
\indent yielding the eigenvalues:
\begin{align*}
\rho_{1,2} =
\frac{9q^2 + 176q + 128}{2(q - 8)(9q - 8)}
\pm
\frac{\sqrt{q\left(81q^3 + 3168q^2 + 30976q + 65536\right)}}{2(9q^2 - 80q + 64)}
\end{align*}
\indent The existence of the multipliers implies that there are two linearly independent solutions, where one exhibits exponential decay and the other exhibits exponential growth. This dynamic signifies that the homogeneous system possesses an exponential dichotomy, a phenomenon that occurs when the solution space admits a decomposition into stable and unstable invariant subspaces. In terms of the system's spectral properties, this corresponds to the case where the multipliers are split by the unit circle:$$|\rho_1| < 1 < |\rho_2| \quad \text{or} \quad |\rho_2| < 1 < |\rho_1|.$$\\
\indent In our study, this dichotomy is the fundamental requirement for the existence and uniqueness of the almost periodic solution. \\
\indent Since it is algebraically difficult to explicitly characterize the set of $q$-values that yield exponential dichotomy, we numerically tested several values and observed the corresponding eigenvalues. Below are some representative $q$ values for which exponential dichotomy is verified:

\[
\begin{array}{c|c|c}
q & |\rho_1| \text{ (Stable)} & |\rho_2| \text{ (Unstable)} \\ \hline
0.10 & 0.5600 & 2.0374  \\
0.20 & 0.4516 & 2.9305  \\
0.30 & 0.3859 & 4.0643  \\
0.40 & 0.3396 & 5.6365  \\
0.50 & 0.3044 & 8.0100 
\end{array}
\]
\indent To illustrate the solution trajectories' behavior quantitatively, we focus our analysis on the case where $q = 0.30$.\\
\indent In this model, one can see that the conditions (C1)-(C4) hold and the associated homogeneous system is exponential dichotomous.\\ 
\indent Thus the solution space of the equation \eqref{eqnHom} decomposes into two subspaces: a stable and an unstable subspace. The long-term behavior of any given solution is strictly determined by the projection of its initial condition onto these subspaces.\\
\indent \indent To demonstrate the exponential dichotomy of the nonhomogeneous system \eqref{eqnNonHom}, we compare two distinct trajectories. The first trajectory is initialized with precisely computed conditions designed to eliminate the unstable components, thereby isolating the bounded, stable behavior driven by the almost periodic forcing term. In contrast, the second trajectory is introduced with a small initial perturbation. Due to the presence of a non-trivial unstable manifold, this slight deviation is rapidly amplified, leading to exponential divergence. This stark separation between the bounded stable trajectory and the divergent trajectory visually confirms the exponential dichotomic behavior of the EPCAG system. These dynamics are clearly observed in Figure \ref{fig:double-trajectory}.
\begin{figure}[H]
    \captionsetup{justification=raggedright, singlelinecheck=false}
    \centering
    \begin{subfigure}[b]{0.48\linewidth}
        \centering
        \includegraphics[width=\linewidth]{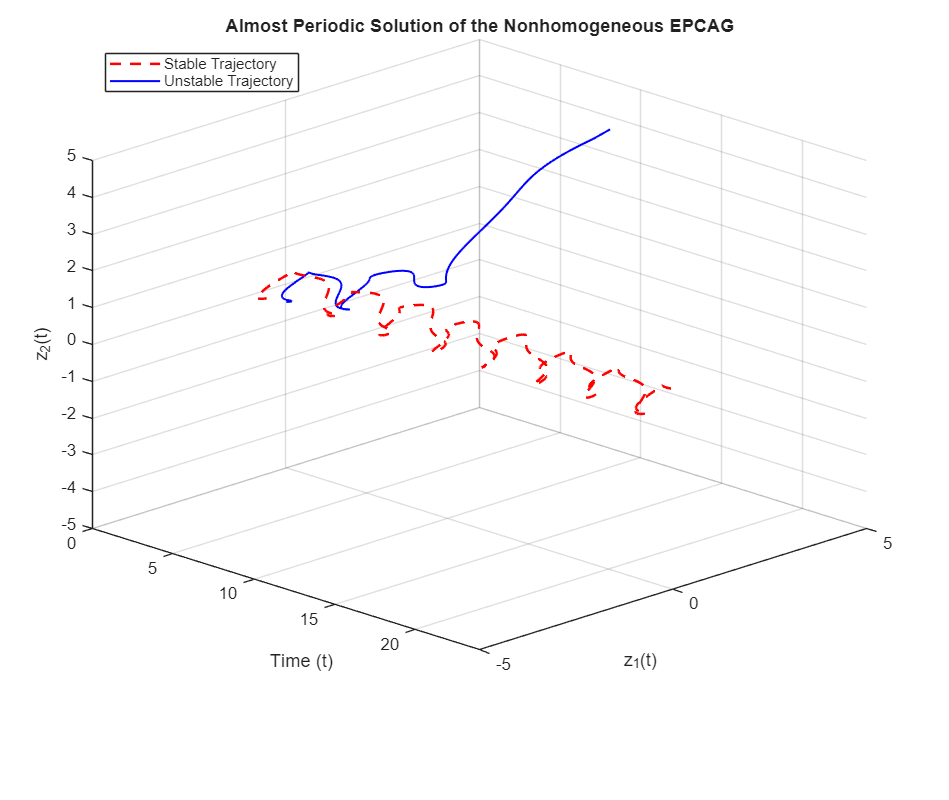}
        \caption*{(a)}
    \end{subfigure}
    \hfill
    \begin{subfigure}[b]{0.48\linewidth}
        \centering
        \includegraphics[width=\linewidth]{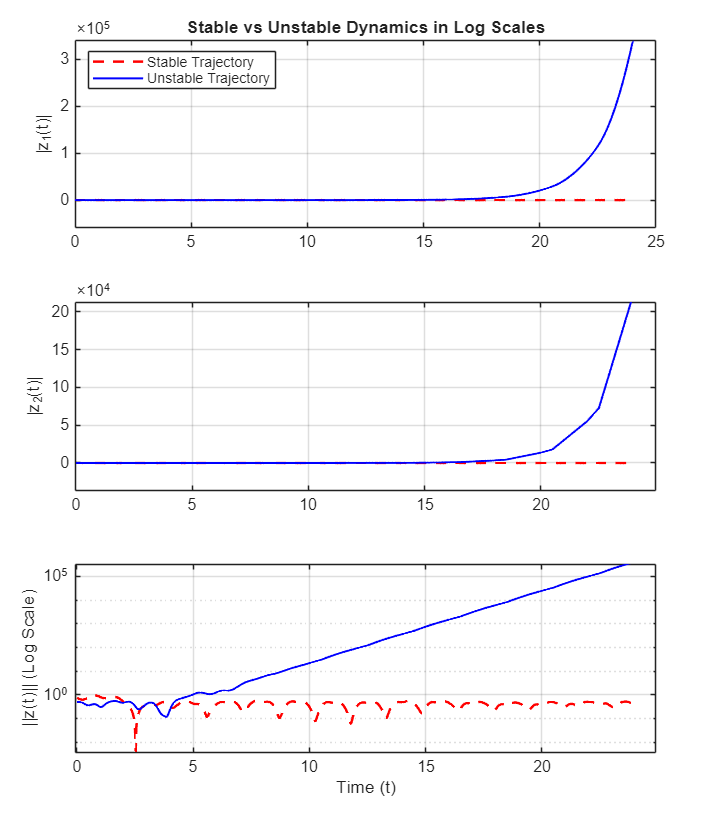}
        \caption*{(b)}
    \end{subfigure}
    \caption{
        Numerical results for the hybrid system \eqref{eqnNonHom} with parameter $q = 0.3$.\\ 
   (a) 3D phase trajectories showing the interaction between the exponential dichotomy and the almost periodic forcing term. The blue curve represents the stable trajectory. Its initial condition is specifically chosen to remove unstable components, ensuring the stable behavior. In contrast, the red dashed curve shows an unstable trajectory starting from a small perturbation, which highlights the rapid exponential divergence.\\
(b) Log-scale plots of $|z_1(t)|$, $|z_2(t)|$, and $\|z(t)\|$ demonstrate the system's exponential dichotomy. The red dashed curve shows the exponential growth of the unstable mode, whereas the blue curve maintains stable behavior throughout the simulation interval. This stability is achieved by computing the initial conditions exactly, which delays the inevitable drift caused by numerical errors.
}\label{fig:double-trajectory}
\end{figure}
These results verify the decomposition of the solution space into stable and unstable parts exactly as the theory predicts, confirming the existence of the unique almost periodic solution within the piecewise constant dynamics of the system.

\section{Conclusions}
\indent \indent Almost periodicity stands as one of the most sophisticated concepts in recurrent dynamics, attracting significant research interest not only for its theoretical depth but also for its versatile functional properties. However, proving the existence and stability of such outputs in hybrid systems remains a formidable challenge. A primary difficulty lies in ensuring that almost periodic properties remain consistent as they transition between continuous and discrete time. Furthermore, describing exponential dichotomies within these systems requires a delicate balance; in our research, we have addressed this by introducing constructive conditions that are directly applicable to practical scenarios.\\
\indent Unlike the established practice of reducing hybrid systems to isolated discrete equations \cite{16,17,18,19,20}, our work provides a more holistic framework. We treat continuity and discontinuity as an inseparable whole \cite{1,2,7}, allowing dichotomic models to be analyzed without fragmenting the system’s natural evolution. This integrated approach is a cornerstone of our investigation, as it preserves the structural integrity and universality of the findings. By maintaining this continuous-time perspective, we bridge the gap between nodal points and inter-interval dynamics—a methodological necessity that has been overlooked in recent literature.\\
\indent Our methodology specifically leverages the exponential dichotomy of the system \eqref{eq:first} to derive precise estimates for the Green’s function. This enables us to capture complex dynamics where stable and unstable components coexist. While the foundations of exponential dichotomy for EPCAG systems were initiated in \cite{7}, we have refined and expanded these concepts to rectify the fundamental oversights found in recent attempts to extend the theory. Specifically, we have demonstrated that a rigorous treatment of dichotomous separation must acknowledge the divergent dynamics of the unstable manifold, thereby resolving structural contradictions such as untenable claims of global stability in the presence of an unstable subspace that persist in other methodologies.\\
\indent In conclusion, by evaluating our integrated approach against traditional techniques, we have clearly highlighted its analytical advantages and robustness. We believe that the framework presented here provides a mathematically sound basis for tackling the interplay between recurrence and complexity. These findings not only settle current theoretical disputes but also offer beneficial tools for a wide range of future applications in hybrid dynamical systems.

\end{document}